\documentclass[dvips,aos,a4]{imsart}

\startlocaldefs

\usepackage{epsfig}
\usepackage{url}
\usepackage{amsmath}
\usepackage{amssymb}
\usepackage{amsthm}
\usepackage{epsfig}
\usepackage{verbatim}
\usepackage{color}
\include{pspicture}

\newcommand{\ninseps}[3]{
\begin{figure}[h]
\begin{center}
 \scalebox{#3}{\includegraphics{#1}}
\end{center}
\caption{\hspace{0.25cm}#2\label{f:#1}}
\end{figure}
}

\newtheorem{theorem}{Theorem}[section]
\newtheorem{lemma}[theorem]{Lemma}

\newtheorem{definition}{Definition}[section]
\newtheorem{remark}{Remark}

\newtheorem{proposition}{Proposition}

\endlocaldefs

\begin{document}

\begin{frontmatter}

\title{Approximation of Bounds on Mixed Level Orthogonal Arrays}
\runtitle{Approximation of Bounds on Orthogonal Arrays}

\author{\fnms{Ferruh} \snm{\"{O}zbudak}\thanksref{m1}\ead[label=e1]{ozbudak@metu.edu.tr}}
\and
\author{\fnms{Ali Devin} \snm{Sezer}\corref{}\thanksref{m2}\ead[label=e2]{devin@metu.edu.tr}}

 \affiliation{Middle East Technical University, Institute of Applied Mathematics\thanksmark{m1}\thanksmark{m2}, Department of Mathematics\thanksmark{m1}}
\runauthor{\"{O}zbudak and Sezer}

\address{Middle East Technical University\\ Eskisehir yolu\\ 06531 Ankara, Turkey\\ 
\printead{e1}\\ \phantom{E-mail:\ }\printead*{e2}}
\begin{abstract}
Mixed level orthogonal arrays are basic structures in experimental design. 
We develop three algorithms that compute Rao and Gil\-bert-Var\-sha\-mov
type bounds for mixed level orthogonal arrays. 
The computational complexity of the
terms involved in these bounds can grow fast 
as the parameters of the arrays increase and this justifies the 
construction of these algorithms.
The first is a recursive algorithm that computes the bounds exactly,
the second is based on an asymptotic analysis and the
third is a simulation algorithm.
They are all based on 
the representation of the combinatorial expressions that
appear in the bounds as expectations involving a
symmetric random walk. The Markov property of the underlying random walk
gives the recursive formula to compute the expectations.
A large deviation (LD) analysis of the expectations
provide the asymptotic algorithm. 
The asymptotically optimal
importance sampling (IS) of the same expectation provides the simulation 
algorithm. Both the LD analysis and the construction of
the IS algorithm uses a representation of these problems as a sequence of
stochastic
optimal control problems converging to a limit calculus of variations problem.
The construction of the IS algorithm uses a recently discovered method
of using subsolutions to the Hamilton Jacobi Bellman equations associated with the limit problem.
\end{abstract}



\end{frontmatter}

\section{Introduction}
Mixed level orthogonal arrays (OAs, for short)  are fundamental to 
experimental design.
Each row of an array 
is thought of as a run of an experiment; each entry of the row is the
value of a parameter of the system being tested. The goal of the experiment
is to test
as wide a range of parameter values of the system as possible. The number
of parameters and which values these parameters can take (i.e., the row length
and the alphabets where the row entries take their values)
are determined by the characteristics of the system being tested. 
The remaining parameters of an OA 
are its number of rows $N$ and its strength $t$. The strength of an OA
is $t$ when the OA is capable of exploring all possible 
interactions of up to $t$ number of parameters of the system (see Definition
\ref{d:oa}).
$N$ is the number of experiments that the OA describes. A high $t$ and
a low $N$ is desirable.
The Rao bound (see \eqref{e:rao1} below), 
first proved for fixed level orthogonal arrays by 
Rao \cite{MR0022821},
gives a lower bound
on $N$ in terms of $t$, the row length and the system parameters (i.e.,
the row length and alphabet sizes). Our first object of study is this bound
and the goal is to develop algorithms that compute exactly and 
approximately the right side of this bound.

The Rao bound is a necessary bound, all OAs satisfy it. 
There are also
sufficient bounds that arise from constructions. One well known construction
method for ordinary OAs is by taking the dual of error correcting codes
\cite{MR1693498}.
\cite{MR2257087} generalizes this idea by 
defining error-block codes, which are error-correcting codes in which one can
specify what alphabet to be used for each entry of the code word.
Furthermore, \cite{MR2257087}  notes that the duals of error block codes
are mixed level orthogonal arrays. 
We use this idea and construction of error block
codes in \cite{MR2350532} to obtain construction of orthogonal arrays
whose parameters satisfy a Gilbert-Varshamov (GV) type bound (see
\eqref{e:gvbound} below).
Our second object of study is this bound.

In subsection \ref{ss:complexity} we calculate the computational complexity
of directly computing the Rao Bound \eqref{e:rao1} and the
GV bound \eqref{e:gvbound}. We see that this complexity
is polynomial in the strength parameter, and the degree of the polynomial
is one more the number of different type of alphabets used in the OA. 
If many different types of alphabets are used in an OA, which is typical
in real life experimental designs, the Rao and the GV bounds become inefficient to compute
directly from their original representations \eqref{e:rao1} 
and \eqref{e:gvbound}.
This potentially high complexity 
of the direct computation of these bounds 
justifies
the construction of new algorithms to compute them. In the present paper
we develop three algorithms for this purpose.
The simple result that underlies these 
is an expectation representation of the Rao and the GV bounds that we derive in 
Sections \ref{s:exp} and \ref{s:GV}.
The expectation is that of a function of a 
random walk whose increments are either
$0$ or $1$ with equal probability. 
The walk takes $n$ steps, the row length of the OA, and  accumulates a cost
throughout its excursion as follows:
if the walk goes up at the
$i^{th}$ step, the accumulated cost increases by a factor one less the
alphabet size of the $i^{th}$ factor of the OA.
The aforementioned representation is the expectation of this accumulated
cost over sample paths which are less than $t/2$ at the
last step of the random walk for the Rao bound and less than $t-1$ 
for the GV bound.

Once these expectation representations are available, it is straightforward
to use them in several ways to obtain algorithms to compute the
bounds. The Markov property of the underlying walk gives
the recursive formula \eqref{e:recursive}. The complexity of this
formula is a second order polynomial in the strength parameter and is
far less than the original formulas when the number of alphabets is large.

The asymptotic behavior of bounds such as the Rao and the GV bounds is a basic 
question to ask. 
\cite{MR2350532} carries out an asymptotic 
analysis of the GV bound for orthogonal arrays with two alphabets. 
To our knowledge, no
results concerning the asympotic behavior of either the GV or the Rao
bound for general mixed level orthogonal arrays is available in the 
current literature.
With our expectation representation
an asymptotical analysis of these bounds becomes
what is called a large deviations analysis (LD) in
probability theory and we use the methods of the LD theory to
carry it out.
In Section \ref{s:ld} we use the stochastic optimal control approach to 
LD \cite{MR512217, MR635802, dupell2}
to show that the right side of the Rao bound \eqref{e:rao1}
grows exponentially in the row length $n$ and identify the growth rate. 
Following \cite{dupell2}, we use a relative entropy representation
of our  expectation  of interest
to write it as a discrete time stochastic optimal control problem. 
Under proper scaling, this control problem
converges to a limit deterministic calculus of variations problem. 
Similar to \cite{thesis,DSW},
the connection
between the prelimit and the limit problems is established using the 
Hamilton Jacobi Bellman (HJB)
equation associated with the limit problem
(see Section
\ref{s:ld} for the Rao bound and in Section \ref{s:GV} for the GV bound).
This analysis provides our second approximation algorithm.
To the authors's knowledge the idea
of using the limit HJB equation to compute large deviation limits first
appeared in \cite{SonerIshiiDupuis} in the context of analysis of queuing
systems.

The asymptotic analysis gives good approximations in an
exponential scale. More accurate approximations can be obtained using
simulation, which is possible because we have the expectation representations
\eqref{e:expectation} and \eqref{e:exp2}.
However, these are expectations over sets with small probability (i.e., rare)
for reasonable values of
the strength parameter $t$.
For such expectations, ordinary simulation would require a great number 
of samples
for reliable estimates. A remedy to this is
importance sampling, which means to change the sampling distribution to
a distribution under which the set over which expectation is taken is not
rare anymore. One modifies the estimator accordingly by multiplying it with
a likelihood ratio to account for the change of the sampling distribution.
IS is a well known idea, it goes back at least to 1949, see, for example,
\cite{Hammersley64,Siegmund,MR1999614, duphui-is1}, and the references
therein.

For our problem, an importance sampling distribution will be one under
which with high probability 
our random walk remains below $t-1$ or $t/2$ at its final step. 
There are many such distributions.
Among these, one would like to choose a distribution that minimizes the variance
of the IS estimator. It is well known that an exact solution of this 
optimization
problem is as difficult as directly
computing the expectation \cite{Goertzel49}. 
In situations such as the one covered in this article where the object
of study is a sequence of expectations decaying or
growing exponentially in
a parameter,
a compromise is to choose a sequence of estimators
whose variance decay or grow exponentially at a rate twice the 
asymptotic decay or growth
rate of the expectation itself. 
Such a sequence is called asymptotically optimal, see \cite{Siegmund}
and \cite{duphui-is1} and the references therein.
To obtain such a sequence we will follow \cite{duphui-is1,DSW} and
represent the variance 
minimization problems in IS once again as
a sequence of stochastic optimal control problems.
Under proper scaling, these also converge to the same limit control problem
as the one that emerges in the large deviations analysis.
Theorem \ref{t:optimal} asserts that a simple change of measure 
based on a piecewise linear subsolution of the HJB equation of the limit control 
problem is asymptotically optimal.
This idea of using subsolutions to construct IS algorithms is from \cite{DSW, thesis, duphui-is3,duphui-is4}
and is called the subsolution approach to IS.

The use of randomized algorithms for counting is one of the central ideas
in statistics. The use of importance sampling for this
purpose seems to be relatively new. \cite{chen2005smc} is the
first article that we are aware of that  uses importance
sampling for purposes of counting. More recent articles since this
work include \cite{chen2006sis,blitzstein2006sis, BlanchetIsCounting}.
The current
work seems to be the first to use the
subsolution method to construct asymptotically optimal IS algorithms for counting.

The plan of the paper is as follows. The next section gives
the definition of an orthogonal array and states the Rao and the GV bounds. It computes
the computational complexity of the original combinatorial 
representation of these bounds. Section \ref{s:exp} derives the expectation 
representation of the
Rao bound and states the exact recursive algorithm to compute it (equation \eqref{e:recursive}).
Section \ref{s:ld} carries out the large deviations analysis of the expectation
representation of the Rao bound. The final result here is Theorem \ref{t:convergence}
with characterizes the growth rate of the bound as a finite dimensional concave
maximization problem. The dimension of the problem is the number of alphabets
used in the OA.  Section \ref{s:is} uses the ideas in the above paragraphs to
construct an asymptotically optimal IS algorithm to estimate the Rao bound, the 
final result is Theorem \ref{t:optimal}. Section \ref{s:GV} does
for the GV bound
what was done for the Rao bound in Sections \ref{s:ld} and \ref{s:is}.
This generalization
requires only minor modifications.  Section \ref{s:numerical} provides
numerical results that gives evidence that the constructed algorithms are effective
in practice as well.

\section{Definitions and Bounds}\label{s:bounds}
We begin with the following definition from \cite{MR1693498}.
\begin{definition}\label{d:oa}
A matrix $A$ is said to be an 
$OA(N, s_1^{l_1} s_2^{l_2}...s_\sigma^{l_\sigma},t)$  if it has the
following structure:
\begin{enumerate}
\item $A$ has $N$ rows,
\item Row length of $A$ is $l_1 + l_2 +\cdots + l_\sigma$;
the first $l_1$ components of each row 
are from the alphabet ${\mathbb Z}_{s_1}$
the second $l_2$ components from ${\mathbb Z}_{s_2}$,..., the last
$l_\sigma$ components from ${\mathbb Z}_{s_\sigma}.$
\item Take any $t$ columns $c_{i_1} c_{i_2}... c_{i_t}$ of $A$ and call
the matrix formed by these columns $A'$. 
Take any string $s$ of length $t$ such that $j^{th}$ letter of $s$
comes from the alphabet corresponding to column $i_j$. Count the number
times $s$ occurs as a row of $A'$. This count is the same for all $s$.
\end{enumerate}
\end{definition}
The last item is the orthogonality property and $t$ is 
the {\em strength} of the orthogonal array. This type of arrays are called mixed
level because the columns are allowed to be from different alphabets (second
property above).

The parameters of any mixed level 
orthogonal array has to satisfy the Rao bound:
\begin{equation}\label{e:rao1}
N \ge \sum_{i=0}^{ t/2 } 
\sum_{
\mbox{
\begin{minipage}{2.4cm}
        \begin{center}\tiny
                        $u_1,u_2,\ldots,u_\sigma\ge 0$\\
        $\sum u_m = i $
        \end{center}
\end{minipage}}} 
\prod_{m=1}^\sigma \left(
\begin{matrix}
        l_m \\
        u_m
\end{matrix}
\right) (s_m - 1 )^{u_m}.
\end{equation}
This bound corresponds to the sphere packing bound for error block codes.
For $\sigma=1$ \eqref{e:rao1} was proved in \cite{MR0022821}, for the proof
of the general case see \cite[page 201]{MR1693498}. 

\subsection{Sufficient bounds}
The duality idea mentioned in the introduction and block error 
code constructions implied by Theorem 3.1 in \cite{MR2350532} 
give mixed level orthogonal
arrays whose parameters satisfy the following conditions:
$s_i = q^{m_i}$
where $q$ is a prime power,
{\small
\begin{equation}\label{e:gvbound}
Nq \ge 
\sum_{i=0}^{ t-1 } 
\sum_{
\mbox{
\begin{minipage}{2cm}
        \begin{center}\tiny
                        $u_1,u_2,\ldots,u_\sigma$\\
        $\sum u_m = i $
        \end{center}
\end{minipage}}} 
s_{\sigma}
\left( \begin{matrix} l_\sigma -1 \\ u_\sigma-1 \end{matrix}\right) 
(s_{\sigma} - 1)^{u_\sigma-1} 
\prod_{m=1}^{\sigma-1} 
\left(
\begin{matrix}
        l_m \\
        u_m
\end{matrix}
\right) (s_m - 1 )^{u_m}
\ge N.
\end{equation}
}
This is a sufficient bound; that is, 
it is known that OA's with these parameters
do exist.
Bounds like \eqref{e:gvbound} are called Gilbert-Varshamov type bounds
in coding theory.

The right side of \eqref{e:gvbound} has essentially the same structure
as that of \eqref{e:rao1}. The key difference between these bounds
is the upper limit of the outer sum: \eqref{e:rao1} goes up to $t/2$ whereas
\eqref{e:gvbound} goes up to $t-1$. 

In the next subsection we will study the computational complexity of
directly evaluating \eqref{e:rao1} or \eqref{e:gvbound}. 

\subsection{Computational complexity of evaluating \eqref{e:rao1} and 
\eqref{e:gvbound}.} \label{ss:complexity}

It follows from their definitions that 
the evaluations of \eqref{e:rao1} and \eqref{e:gvbound} have
the same computational complexity. Therefore, it is enough to consider
one of them.

The right side of \eqref{e:rao1} involves a partitioning
of each $i$ less than $t/2$ into a sum of $\sigma$ integers. 
The number of such partitions 
is
$
\left( \begin{matrix}  \sigma +  i-1 \\ \sigma -1 \end{matrix} \right).
$
Then the number of operations needed to compute the right side of
\eqref{e:rao1} is bounded below by
\begin{equation}\label{e:complex}
	\sum_{i=0}^{t/2}  
	\sigma \left( \begin{matrix} \sigma + i  -1 \\ \sigma -1
	\end{matrix} \right)\ge \sum_{i=0}^{t/2 -1} i^{\sigma}
	\ge C (t/2)^{\sigma + 1},
\end{equation}
where $C$ is a constant that depends only on $\sigma$.
If the strength parameter $t$ grows linearly in $n$, i.e., if $t = \mu n$,
where $\mu \in (0,1)$,
a direct computation of \eqref{e:rao1} requires
$O(n^{\sigma+1})$ operations.
The present paper is aimed
at finding methods to compute
\eqref{e:rao1} and \eqref{e:gvbound} more efficiently.

The next section presents a simple probabilistic representation of 
\eqref{e:rao1}, which forms the basis for all of the results and algorithms
presented in this paper.

\section{Expectation Representation}\label{s:exp}
Let $X_i$ be independent and identically distributed (iid)
Bernoulli random variables with $P(X_i = 1) = P(X_i =0) =1/2.$
Let $S_k \doteq X_1 + X_2 + \cdots + X_k.$ Define the following ``running cost:''
$$
r(x,j) \doteq
\begin{cases}
	1, &~~\text{ if } x =0 \\
	(s_i-1),&~~\text{ if }  x =1, \text{ and } ~~\sum_{k=1}^{i-1}l_k +1 \le j \le 
\sum_{k=1}^{i} l_k.
\end{cases}
$$
\eqref{e:rao1} 
can be written in the form
\begin{equation}\label{e:expectation}
N  \ge 
2^n 
{\mathbb E}\left[ 1_{\{S_n \le t/2\}}
\prod_{j=1}^{n} r(X_j,j) \right] = 
{\mathbb E}\left[ 1_{\{S_n \le t/2\}}
\prod_{j=1}^{n} 2r(X_j,j) \right].
\end{equation}
This is an expectation over the trajectories of $S_k$
that stay below the level $t/2$ at step $n$. At each step the random walk
accumulates a running cost $r$; the cost depends on the step number and
the current step.
The random walk can be thought of as a scan of the letters of a row of the
array.
At each step we flip a coin to decide whether the current letter will
be included in the computation. If the decision is yes, i.e., if $X_i=1$
and the random walk goes up, then the current
bound is multiplied with $2(s_i -1)$ 
where $s_i$ is the alphabet size of the
letter we are going over (this is the $2r$ term in \eqref{e:expectation}).
The first sum in \eqref{e:rao1} group trajectories according to their
positions at step $n$. For a position $i \le t/2$,
the second sum in \eqref{e:rao1} partition these $i$ up-steps into different cost regions and the binomial
coefficients count the number of possible ways $u_m$ up-steps can be taken
in $l_m$ steps.

\paragraph{A simple recursive algorithm to compute the Rao bound}
Our first method to compute \eqref{e:rao1} is a recursive algorithm
that computes the bound exactly.
For integers $0 \le x\le t/2$ and $0\le k \le n$
define
$$
M(x,k) = {\mathbb E}\left[ 1_{\{ x + S_{n-j} \le t/2 \} } \prod_{j=k+1}^n
2r(X_i, j) \right].
$$
The Rao bound \eqref{e:rao1} in terms of $M$ is $N \ge M(0,0).$ 
Because $X_i$ are iid and $S_i$ are their sum, $M$ satisfies
\begin{equation}\label{e:recursive}
M(x,k) = M(x+1,k+1) r(x,k) + M(x,k+1),
\end{equation}
for $x < t/2$ and $ k < n$.
In addition, we have the boundary conditions $M(x,n) = 0$ for $x \le t/2$
and $M(t/2, k) = 0$ for $ k \le n$.
These give an algorithm that takes only $t n /2$ steps to compute the
Rao bound.  If we write the strength parameter $t$ as a fraction $\mu$ of $n$ as
$t = \mu n$ then the complexity analysis in the previous chapter implies
that the direct evaluation of \eqref{e:rao1} will take at least 
$O(n^{\sigma+1})$  operations.
Whereas the computation of the same bound 
using \eqref{e:recursive} will only take
$O(n^2)$ operations.

\section{Large Deviations Analysis}\label{s:ld}
The goal of this section 
is an asymptotic analysis of the right side of \eqref{e:expectation}
as $n\rightarrow \infty$. In order for this analysis to be meaningful $t$
and $l_i$ need to grow with $n$ as well. Therefore we assume that
\begin{equation}\label{e:scalingassumptions}
t = \mu n,~~ \mu \in (0,1),  ~~l_i = na_i, ~~ \sum a_i = 1.
\end{equation}
The asymptotic analysis of \eqref{e:rao1} now consists of evaluating
\begin{equation}\label{e:LDlimit}
\lim_n \frac{1}{n}\log
{\mathbb E}\left[ 1_{\{S_n \le t/2\}}
\prod_{j=1}^{n} 2r(X_j,j) \right].
\end{equation}

For the evaluation of \eqref{e:LDlimit}, we will follow \cite{dupell2} and
begin by representing the
$\log{\mathbb E}[\cdots]$ term in it as a discrete time stochastic
optimal control problem as follows.
\begin{proposition}\label{p:relentrep}
The following identity holds:
\begin{align}\label{e:relentrep}
&\log  {\mathbb E}\left[ 1_{\{S_n \le t/2\}}
\prod_{j=1}^{n}2r(X_j,j) \right] \\
&~~=
\sup_{
\mbox{
\begin{minipage}{2cm}
        \begin{center}\tiny
 $\bar{p}(\cdot|\cdot,\cdot)$\\
~\\
$\bar{P}(S_n\le\frac{ \mu n}{2}) = 1$\notag
        \end{center}
\end{minipage}}} 
\hspace{-0.5cm}
\bar{\mathbb E}\left[
\sum_{j=1}^n \log r(X_j,j ) - \log \bar{p}(X_{j}|j,S_{j-1})\right],
\end{align}
where the $\sup$ is over all transition probabilities
$\bar{p}(\cdot|\cdot,\cdot):{\mathbb Z}_2 \times {\mathbb N} \times {\mathbb N}$
$\rightarrow$ $[0,1]$
that give the probability of the steps $0$ and $1$ given the current position of
and the current step number of the random walk $S$ and $\bar{P}$ is the probability
distribution defined by these measures on the path space of the random walk.
\end{proposition}
The proof of this result is similar to that of Proposition 1.4.2 
\cite[page 31]{dupell2} and is omitted.
The sup on the right side of \eqref{e:relentrep} over all Markov chains
on the sample paths of $S_k$ such that the $n^{th}$ step is less than $t/2$
with probability $1$. The $\log$ term inside the sup corresponds to the entropy of
 $\bar{p}(\cdot|\cdot,\cdot)$.

Define $A_i \doteq \sum_{j=1}^i a_j$ and
$$
\tilde{r}(t) \doteq \log(s_i -1 ), ~~~ A_i \le t < A_{i+1},
$$
and let
$$H(\theta) \doteq -\theta \log\theta - (1-\theta) \log(1-\theta)$$
be the entropy function.
As we observed earlier, the right side of \eqref{e:relentrep} is a 
stochastic optimal control problem.
Upon dividing it by $n$ and scaling the time and space parameters with  
$\frac{1}{n}$, and sending $n$ to $\infty$ one obtains the following
limit deterministic optimal control problem:
\begin{equation}\label{e:limit}
\sup_{\theta(\cdot)} \int_0^1 [\tilde{r}(t)\theta +  H(\theta)]dt,
\end{equation}
where the $\sup$ is over all measurable functions on $[0,1]$ with values in
$[0,1]$ such that $\int_0^1 \theta(t)dt \le \mu/2$.
The rigorous connection between this optimal control problem and \eqref{e:relentrep} can be established in several ways. For example, one can use the
weak convergence approach of
\cite{dupell2}. Another approach is via the HJB
equation associated with the limit control problem \eqref{e:limit}
and a verification argument, which is followed in \cite{DSW}. 
In this paper we will take this second path because the same method will 
also allow us to prove the asymptotic optimality of an IS estimator based
on a subsolution of the limit HJB equation.

\subsection{Solution to the limit control problem}
For $ A_i \le t \le A_{i+1}$,
$L(t,\theta) = \log(s_i -1)\theta + H(\theta)$ is a strictly
concave function with no $t$ dependence. Then Jensen's inequality implies that
the optimal trajectory needs to be a straight line between times
$A_i$ and $A_{i+1}$. Therefore, it is enough to consider the optimization
problem \eqref{e:limit} over piecewise linear continuous paths and
the $\sup$ in  \eqref{e:limit} equals
\begin{equation}\label{e:limitfinite}
\sup 
\left\{ \sum_{i=1}^\sigma 
a_i\left( \theta_i \log (s_i -1)  + H(\theta_i) \right)\right\},
\end{equation}
where the $\sup$ is subject to
\begin{equation}\label{e:orconst}
	\theta_i \in (0,1), ~~ \langle a , \theta \rangle = \mu/2.
\end{equation}
The objective function of this finite dimensional constrained optimization problem
is
strictly concave and its constraints linear. 
Therefore, a straightforward use of a Lagrange multiplier 
converts the problem to a one of root finding of a one dimensional
monotone function.

In the next subsection we will prove that a function defined based on
\eqref{e:limitfinite} satisfies an HJB equation.
We will use this fact to prove the convergence of \eqref{e:LDlimit} to
\eqref{e:limitfinite}.

\subsection{The limit Hamilton Jacobi Bellman equation}
Let us generalize the problem in \eqref{e:limit} so that the problem
starts
from any initial point $x \le \mu/2$ at any time $t \in [0,1]$:
\begin{equation}\label{e:defV}
V(x,t) = \sup_{\theta} \int_t^1 [\tilde{r}(s) \theta(s) + H(\theta(s)) ]ds,
\end{equation}
where the $\sup$ is over all measurable $\theta(\cdot) \ge 0$ such that 
$x + \int_{t}^1 \theta(s)ds \le \mu/2$.
The $\sup$ in \eqref{e:limit} equals $V(0,0)$. 
Generalizing \eqref{e:limitfinite},
for $A_i \le t <  A_{i+1}$ 
we have that
{\small
\begin{align}\label{e:finitext}
	&V(x,t)\\
&~\doteq \sup\left\{ (A_{i+1} - t )(\theta_{i} \log(s_i - 1 ) + H(\theta_i) +  \sum_{j = i+1 }^\sigma a_j ( \theta_j \log(s_j - 1) + H(\theta_j) ) \right \},\notag
\end{align}
}
where the $\sup$ is subject to
\begin{equation}\label{e:constraints}
	\theta_j \in (0,1), ~~ x + \theta_i (A_{i+1}-t) + \sum_{j=i+1}^\sigma a_j \theta_j \le \mu/2.
\end{equation}
Let us now write $V$ more explicitly. Firstly, the absolute maximizer
of \eqref{e:finitext} without the constraints \eqref{e:constraints} is
\begin{equation}\label{e:absmax}
\theta_j^* = \frac{s_j  -1 }{s_j}.
\end{equation}
If $\theta_j^*$ satisfy \eqref{e:constraints}, i.e., if
\begin{equation}\label{e:absmaxcond}
	x + \theta_i^* (A_{i+1} -t) + \sum_{j=i+1}^\sigma a_j \theta_j^* \le \mu/2
\end{equation}
then $V$ equals
\begin{align*}
V(x,t) &= (t - A_i) \left[ \left( \frac{s_i  - 1 }{s_i}\right) \log(s_i - 1) + 
H( (s_i-1 ) / s_i ) \right]\\
&~~~~~+ 
\sum_{j=i+1}^\sigma 
a_j  \left[ \left( \frac{s_j  - 1 }{s_j}\right) \log(s_j - 1) + 
H( (s_j-1 ) / s_j ) \right].
\end{align*}
If the absolute maximizers \eqref{e:absmax} do not satisfy \eqref{e:constraints}
then one can use a Lagrange multiplier $\lambda$ to solve \eqref{e:finitext}:
\begin{align*}
 \log(s_j - 1) + \log\frac{(1 - \theta_j ) }{\theta_j} &= 
\lambda,~~~~ j \ge i.
\end{align*}
Then 
\begin{equation}\label{e:explicit}
\theta_j^*(\lambda) =  \frac{s_j-1}{e^\lambda + s_j - 1 }.
\end{equation}
For these to give a solution to \eqref{e:finitext} they must satisfy 
\eqref{e:constraints}:
\begin{equation}\label{e:const2}
	(A_{i+1} - t ) \frac{s_i -1 } {e^\lambda + s_i - 1 } 
+ \sum_{j=i+1}^\sigma a_j \frac{s_j - 1 } {e^\lambda + s_j - 1 } 
=\mu/2 - x.
\end{equation}
For $\lambda = 0$, the left side is by assumption greater than $\mu/2 - x$ and
for $\lambda = \infty$ it is $0$. Because it is monotone in $\lambda$, there
exists a unique $\lambda^*(t,x)$ for which \eqref{e:const2} is satisfied.
By the implicit function theorem, $\lambda^*(t,x)$ is twice 
differentiable in both $t$ and $x$ with bounded derivatives for $
t \neq A_j$. And for $t = A_j$, $\lambda$ has  right derivatives in $t$
and an ordinary derivative in $x$.
Because the function that is optimized in \eqref{e:finitext} is strictly
concave, the stationary point given by $\lambda^*$ is actually a global
maximizer.

Define
\begin{align*}
\tilde{V}(t,\lambda )  &\doteq
( A_{i+1} -t ) \left[ \left( \frac{s_i  - 1 }{e^\lambda + s_i - 1}\right) 
 \log(s_i - 1) + 
 H\left( \frac{s_i-1 }{e^\lambda + s_i - 1 }\right) \right]\\
&~~~~~+ 
\sum_{j=i+1}^\sigma 
a_j  \left[ \left( \frac{s_j  - 1 }{e^\lambda + s_j - 1}\right) \log(s_j - 1) + 
H\left( \frac{ s_j-1 }{ e^\lambda + s_j -1 }\right) \right].
\end{align*}
In light of the above computations, 
$V(x,t)$ of \eqref{e:finitext} can be written more
explicitly as
$$
V(x,t) = \begin{cases} 
	\tilde{V}(t,0),& ~~~ \text{ if \eqref{e:absmaxcond} holds},\\
	\tilde{V}(t,\lambda^*(x,t) ),& ~~~ \text{ otherwise}.
\end{cases}
$$

One obtains the following
proposition by ordinary calculus and implicit differentiation.
\begin{proposition}\label{p:regularity}
$V$ is smooth except for $t=A_i$ where it has 
directional derivative $V_t(x,t)$ which is 
defined as $V_t(x,t) = \lim_{h \searrow 0 } ( V(x,t+h) - V(x,t))/h$.
Higher order $t$ partial derivatives similarly exists.
In particular for any $t$ and $x$ we have:
$$
V(x+\delta, t+h) = V(x,t) + \delta V_x(x,t) + h V_t(x,t) + 
c(x,t)(\delta^2 + h^2)
$$
where $\sup_{x,t} |c(x,t)| =C < \infty $.
\end{proposition}
Now we state the HJB equation satisfied by $V$.
\begin{theorem}\label{t:HJB}
The following dynamic programming equation holds:
\begin{equation}\label{e:HJB}
0 = \sup_{\theta \in [0,1] } \{ \tilde{r}(t) \theta + H(\theta) + V_x(x,t) \theta
+ V_t(x,t) \}
\end{equation}
for 
$(x,t) \in [0,\mu/2) \times [0,1)$.
\end{theorem}
\begin{proof}
	Take $(x,t) \in [0,\mu/2) \times [0,1)$, a small $\delta > 0$ and
	$\theta \in [0,1].$
	\eqref{e:defV} implies
	\begin{align*}	
	V(x,t) &\ge \int_t^{t+\delta} \tilde{r}(s) \theta + H(\theta) ds
	 + V(x + \theta \delta , t+ \delta)\\
	 V(x,t) - V(x + \theta \delta, t + \delta) &\ge
	 [\log(s_i -1 ) + H(\theta) ] \delta.
 \end{align*} 
 Because $V_t$ and $V_x$ exist, dividing both sides of the last display
 by $\delta$ and letting $\delta \rightarrow 0$ gives:
 $$
 -V_t - \theta V_x \ge \log(s_i - 1) + H(\theta).
 $$
Because this is true for all $\theta \in [0,1]$ we have:
$$
0 \ge \sup_{\theta \in [0,1] } \{ \tilde{r}(t) \theta + H(\theta) + 
V_x(x,t) \theta
+ V_t(x,t) \}
$$
One replaces $\ge$ with $=$ by taking $\theta$ to be the optimal
control $\theta^*(\lambda^*(x,t) ).$
\end{proof}

\subsection{Convergence Analysis}
In this subsection we formally
connect the sequence of stochastic optimal control problems
in \eqref{e:relentrep} to the limit control problem \eqref{e:limit}
and its solution developed in the previous subsection.

Figure \ref{f:limit} gives the level curves of $V(x,t)$ and 
$V_{60}(\lfloor n x \rfloor ,\lfloor n t\rfloor )$ 
where
$$
V_n(x,i) = \frac{1}{n} \log 
{\mathbb E}\left[ 1_{\{x + S_n \le \mu n  /2\}}
\prod_{j=i}^{n} 2r(X_j,j) \right] 
$$
for $ a_1 = a_2 = a_3 = 1/3$, $s_1 = 2$, $s_2 =30$, $s_3 = 100$ and
$\mu= 0.1.$
This figure suggests that $V_n(nx, nt ) \rightarrow V(x,t)$
for all values of $(x,t)$. Our main convergence theorem, which we state 
and prove next,
concerns the special case when
$(x,t) =(0,0)$.

\begin{theorem}\label{t:convergence}
	The large deviations limit in \eqref{e:LDlimit} equals $V(0,0)$,
	i.e.,
	\begin{equation}\label{e:convergence}
\lim_n \frac{1}{n}\log 2^n
{\mathbb E}\left[ 1_{\{S_n \le t_n /2\}}
\prod_{j=1}^{n} r(X_j,j) \right] = 
\sup 
\left\{ \sum_{i=1}^\sigma 
a_i\left( \theta_i \log (s_i -1)  + H(\theta_i) \right)\right\},
\end{equation}
where the $\sup$ is over 
\begin{equation}\label{e:constraintst0}
	\theta_i \in (0,1), ~~ \langle a , \theta \rangle = \mu/2.
\end{equation}
\end{theorem}
\ninseps{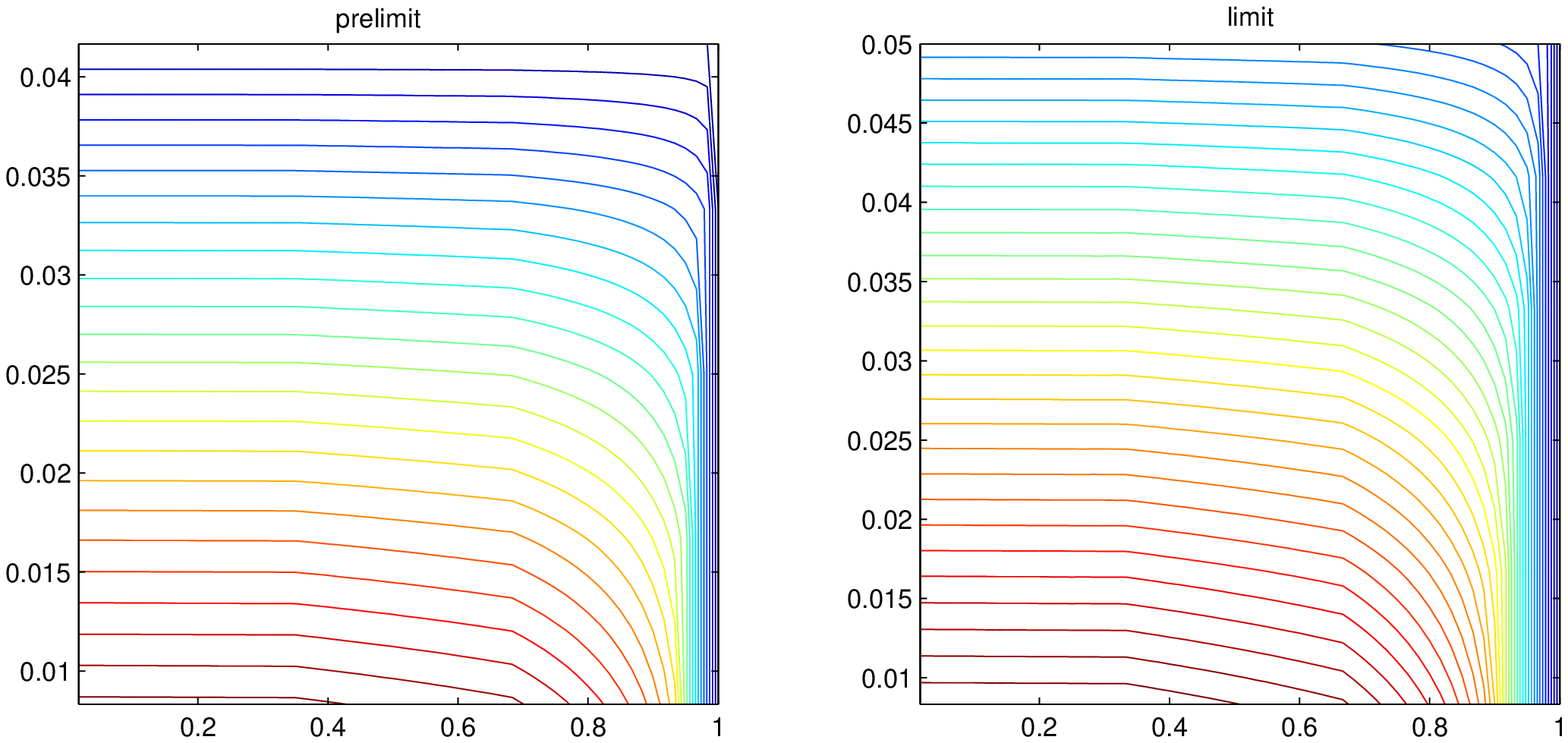}{The level curves of $V(x,t)$ and $V_{60}$ }{0.4}

\begin{proof}
	The proof will be a verification argument using
	$V$ and the HJB equation \eqref{e:HJB}.
	By Proposition \ref{p:relentrep} there exists $\bar{p}^n(\cdot|\cdot,\cdot)$ such that
	$$
	\log {\mathbb E}
	\left[1_{\{S_n \le t/2 \}} \prod_{j=1}^n 2r(X_j,j) \right]
	=
	\bar{\mathbb E}\left[ \sum_{j=1}^n r(X_j,j) - 
	\log \bar{p}^n(X_j|j,S_{j-1}) \right] + \epsilon(n)
	$$
	where
	$\epsilon(n) \rightarrow 0$ and $\bar{\mathbb E}$ is expectation
	with respect to $\bar{p}^n(\cdot| \cdot, \cdot) $.

	\begin{align*}
	&V(0,0) = \bar{\mathbb E}[ V(0,0) - V(S_n/n,1) ] \\
	&~~~~=\sum_{j=0}^{n-1} \bar{\mathbb E}[  V(S_j/n,j/n) - V(S_{j+1}/n,(j+1)/n)]
	\intertext{By Proposition \ref{p:regularity} this equals}
	&~~~~= \frac{C(n)}{n} + 
	\sum_{j=1}^n \bar{\mathbb E}\left[ -V_x(S_j/n,j/n) X_j/n  -  V_t(S_{j+1}/n,(j+1)/n)\frac{1}{n}\right],
	\intertext{where $\sup_n |C(n)| < \infty$. 
	One can condition the last expectation on $S_{j}$ to rewrite it as}
	&~~~~=\frac{C(n)}{n}  + \frac{1}{n}\sum_{j=1}^n \bar{\mathbb E}\left[ -V_x(S_{j}/n,j/n) \bar{p}^n(1|S_{j}/n,j/n)   - 
	V_t(S_{j+1}/n,(j+1)/n)\right].
	\intertext{Now by Theorem \ref{t:HJB} this last sum is greater than:}
	&~~~~~\ge \frac{C(n)}{n}  + \sum_{j=1}^n \bar{\mathbb E}\left[   \tilde{r}(    j/n ) \bar{p}^n(1|S_j/n,j/n)-   H(\bar{p}^n(1|S_j,j )  \right]
\intertext{which in turn equals}
&~~~~~= \frac{C}{n}  + \frac{1}{n}\sum_{j=1}^n \bar{\mathbb E}\left[    r(j, X_j)   -\log(\bar{p}^n(X_j|S_j,j) ) \right]
\end{align*}
Letting $n$ go to infinity yields
\begin{align*}
V(0,0) &\ge \limsup \frac{1}{n} \bar{\mathbb E}^n
\left[\sum_{j=1}^n   r(j, X_j)   -\log\bar{p}^n(X_j|S_j,j)  \right] \\
&\ge \limsup \frac{1}{n}\log  {\mathbb E}\left[ 1_{\{S_n \le t_n/2 \}} \prod_{j=1}^n 2r(X_j,j)\right] -\epsilon.
\end{align*}

For the reverse inequality we first note that the result of the optimization in
\eqref{e:convergence} is continuous in the strength parameter $\mu$ which appears in the constraint 
\eqref{e:constraintst0}.
Let $\theta^*_i $ be the optimizers of \eqref{e:convergence} when the $\mu$ in \eqref{e:constraintst0}
is replaced with $\mu-4\epsilon$ where $\epsilon > 0$ is a small constant. 
Let 
\begin{equation}\label{e:optimalcm}
	\bar{p}^*(1|x,j) = \theta^*_i \text{ if }  A_i \le j/n \le A_{i+1} 
\end{equation} and $\bar{P}^*$ be the measure
on the path space of $(S,X)$ corresponding to $\bar{p}^*$. 
We would like to use $\bar{P}^*$ on the right side of \eqref{e:relentrep} to get a lower bound
on its left side. Once this is done the law of large numbers would give us the bound we desire.
The only problem is $\bar{P}^*(S_n \le \mu n / 2 ) < 1$ so $\bar{P}^*$ is not
included in the set of measures over which the right 
side of \eqref{e:relentrep} is optimized. This is a minor technical
problem and can be handled as follows.
By definition $X_j$ is iid for $ A_i \le j/n <A_{i+1}$.
Therefore the ordinary law of large numbers is applicable and gives:
$$
\bar{P}^*(S_n/n > \mu/2-\epsilon ) \rightarrow 0.
$$
Then, the fact that $\bar{P}^*( S_n \le \mu n / 2) \neq 1$ is not a major problem
and can be dealt with by simply conditioning it on $\{ S_n \le \mu n / 2 \}$.

The details of this argument is as follows. 
Let $p^*_n = \bar{P}^*( S_n \le \mu n / 2)$ and
define 
\begin{equation}\label{e:defcondP}
	\bar{P}^{*,c} =\frac{1}{p^*_n } 1_{\{S_n \le \mu n \/2\}}\bar{P}^*.
\end{equation}
Under $\bar {P}^{*,c}$, $(X_n,S_n)$ is a Markov chain whose transition probability is 
\begin{equation}\label{e:ctrans}
\bar{p}^{*,c}( 1 | s,j) = \bar{p}^*(1|s,j) \frac{ \bar{P}^*(S_n \le \mu n / 2  | S_{n-j} = s+1) }{\bar{P}^*(S_n \le \mu n /2
| S_{n-j-1} = s)}.
\end{equation}
$\bar{P}^{*,c}( S_n \le \mu n /2 ) = 1$ and therefore by Proposition \ref{p:relentrep} we have:
\begin{align*}
	&\frac{1}{n} \log 2^n {\mathbb E}\left[ 1_{\{S_n \le t/2\}}
\prod_{j=1}^{n} r(X_j,j) \right] \\
&~~~~\ge
\frac{1}{n}
\bar{\mathbb E}^{*,c} \left[
\sum_{j=1}^n \log r(X_j,j ) - \log \bar{p}^{*,c} (X_{j}|j,X_{j-1})\right]
\intertext{By \eqref{e:ctrans} and \eqref{e:defcondP} this equals}
&~~~~=
\frac{1}{p^*_n}
\frac{1}{n}
\bar{\mathbb E}^{*} \left[ 1_{\{S_n \le \mu n / 2 \} } 
\sum_{j=1}^n \log r(X_j,j ) - \log \bar{p}^{*} (X_{j}|j,X_{j-1})\right]
+ \log p^*_n
\intertext{ $\theta^*_i$ and $r$ are all positive and bounded. Therefore there exists a positive
$C$ such that }
&~~~~\ge
\frac{1}{p^*_n}
\frac{1}{n}
\bar{\mathbb E}^{*} \left[
\sum_{j=1}^n \log r(X_j,j ) - \log \bar{p}^{*} (X_{j}|j,X_{j-1})\right]
+ \log p^*_n - \frac{C (1-p^*_n) }{p^*_n}\\
&~~~~=
\frac{1}{p^*_n}
\frac{1}{n}
\sum_{i=1}^\sigma (l_i\theta^*_i \log(s_i -1 ) + H(\theta^*_i)  )
+ \log p^*_n - \frac{C (1-p^*_n) }{p^*_n}
\end{align*}
By the law of large numbers $p^*_n \rightarrow 1$. This, the last sequence
of inequalities
and the definition of $l_i$ give:
\begin{align*}
&\liminf \frac{1}{n} \log 2^n {\mathbb E}\left[ 1_{\{S_n \le t/2\}}
\prod_{j=1}^{n} r(X_j,j) \right] \\
&~~~~\ge 
\sum_{i=1}^\sigma (a_i\theta^*_i \log(s_i -1 ) + H(\theta^*_i)  ) = V(0,0) - \delta
\end{align*}
where $\delta$ is a small number that goes to $0$ with $\epsilon$. This inequality
concludes the proof of this theorem.
\end{proof}

\section{Importance Sampling}\label{s:is}
The expectation representation 
\eqref{e:expectation} of the Rao bound brings to mind the possibility of
estimating it using simulation. 
Because the strength parameter $t$ is usually a fraction of $n$
and because the aforementioned expectation is over the set  $\{S_n\le t/2\}$,
for large values of $n$
a direct simulation would require too many samples of $S_n$ to converge. 
One can instead use
{\em importance sampling}, which means to sample from a new
simulation measure under which $\{S_n\le t/2\}$ is not rare. 
The samples
are scaled by the Radon Nikodym derivative of the original measure with
respect to the new sampling measure so that the simulation algorithm still
estimates the probability under the original measure. The main problem
in IS is the choice of the new sampling distribution. One tries to choose
it so that it is practical to sample from it and that
it nearly minimizes estimator variance.
In the next subsection we briefly introduce the main ideas 
of IS in a general setting before we focus on its use in our current setup.

IS is a well known method for estimating small probabilities,
a very partial list of articles and books on the subject
are
\cite{Goertzel49,Hammersley64,  Siegmund, ParWal, MR1999614,duphui-is1,DSW,thesis}.
These works contain many more references to important works on the subject.

\subsection{IS Review}
Take a probability space $(\Omega, {\mathcal F}, P)$ and a measurable
integrable function $f: \Omega \rightarrow {\mathbb R}.$  Suppose  $\hat{P}$ is
a probability measure on $(\Omega, {\mathcal F})$ 
with respect to which $P$ is absolutely continuous.
We have the following basic identity:
\begin{equation}\label{e:isgeneral}
{\mathbb E}[f] = \int_{\Omega } f(\omega)dP(\omega) = 
\int_{\Omega} f(\omega) \frac{dP}{d\hat{P}}(\omega) d\hat{P}(\omega) = \hat{\mathbb E}\left[ f \frac{dP}{d\hat{P}} \right],
\end{equation}
where $\frac{dP}{d\hat{P}}$ is the Radon Nikodym derivative of $P$ with respect to $\hat{P}$ and
${\mathbb E}$ [$\hat{\mathbb E}$] is expectation with respect to $P$ [$\hat{P}$].
The identity 
\eqref{e:isgeneral} suggests the following simulation algorithm to compute ${\mathbb E}[f]$. 
Simulate iid copies $\omega_1$, $\omega_2$,\ldots, $\omega_N$ of $\omega$ from $\hat{P}$ and use the following
to estimate ${\mathbb E}[f]$:
$$
\hat{s}_N = \frac{1}{N} \sum_{i=1}^N \hat{f}(i), ~~~ \hat{f}(i) \doteq f(\omega_i) \frac{dP}{d\hat{P}}(\omega_i).
$$
By the law of large numbers $\hat{s}_N \rightarrow \hat {\mathbb E}\left[ f \frac{dP}{d\hat{P}} \right]$
which by \eqref{e:isgeneral} equals ${\mathbb E}[f]$. Furthermore by the linearity of the expectation
and \eqref{e:isgeneral} one also has $\hat{\mathbb E}[\hat{s}_N] =\hat{\mathbb E}[\hat{f}(1)]= {\mathbb E}[f].$ Therefore $\hat{s}_N$
is an unbiased estimator of ${\mathbb E}[f]$ that converges to this values as $N \rightarrow \infty$. This
method of estimating ${\mathbb E}[f]$ is called importance sampling (IS).

IS is especially useful when $P(\{ f \neq 0 \})$ is small. In such a case,
ordinary Monte Carlo will require a large number of samples $\omega_i$
for a reliable estimate of ${\mathbb E}[f].$ One could choose
$\hat{P}$ so that $\hat{P}(\{ f \neq  0\})$ is no longer small
and hope to reduce the number of samples required for a good estimate. 
A simple way to choose $\hat{P}$ so that this happens is to minimize
the variance of the IS estimator $\hat{s}_N$.
Because $\hat{s}_N$
is unbiased its variance depends on $\hat{P}$ only through
its second moment which equals $\hat{\mathbb E}[\hat{f}^2(1)]/N$. Here
$N$ is the number of samples used in the estimation and is taken to 
be a constant.
Therefore, for a good IS estimator one tries to 
solve the following optimization problem:
\begin{align}\label{e:generalisinf}
	\inf_{\hat{P}} \hat{\mathbb E}[ \hat{f}^2(1)] = 
\inf_{\hat{P}} \hat{\mathbb E} \left[ \left(f \frac{dP}{d\hat{P} } \right)^2 \right] &= 
\inf_{\hat{P}} {\mathbb E} \left[ f^2 \frac{dP}{d\hat{P} }  \right],
\intertext{where the $\inf$ is over all $\hat{P}$ with respect to which $1_{\{f \neq 0\} }dP$
is absolutely continuous.
The exact solution to this problem turns out to have an easy description.
It is simple to prove that $d\hat{P}^* = \frac{f}{ \mathbb{E}[f] } dP$
is actually the minimizer of \eqref{e:generalisinf} and hence we have that
}
&={\mathbb E} \left[ f^2 \frac{dP}{d\hat{P}^* }  \right]= ({\mathbb E}[f])^2.
\notag
\end{align}
Then $({\mathbb E}[f])^2/N$ is the smallest possible second moment for
an IS estimator which uses $N$ samples and
the estimator defined by $\hat{P}^*$ has zero variance.

\subsubsection{Asymptotic Analysis}
As is well known in the IS literature,
$\hat{P}^*$ is not a practical simulation
measure because knowing it requires knowing ${\mathbb E}[f]$
which is the very quantity that is not known and whose estimation is sought.
Therefore, one usually
seeks an almost minimizer 
of \eqref{e:generalisinf} to conduct a good IS simulation. If 
there is a sequence $f_n$ whose expectation is sought,
one way to find almost minimizers to \eqref{e:generalisinf} is to conduct an asymptotic
analysis of the sequence of optimization problems given by \eqref{e:generalisinf}
and the sequence $f_n$. If these converge in some sense to 
a relatively simple limit problem then the optimizers of
this limit problem can inform the construction of almost minimizers to \eqref{e:generalisinf}
for the estimation of ${\mathbb E}[f_n]$. 

One setup where such an asymptotic analysis is possible is when the underlying
measure $P$ is that of a Markov process and
\begin{equation}\label{e:ldlimit}
\lim \frac{1}{n}\log {\mathbb E}[f_n] \doteq \gamma
\end{equation}
exists and is nonzero. As the reader have already seen in the previous
section, the problem in this article
falls into this category.

When the limit \eqref{e:ldlimit} exists, one can define an asymptotic
optimality condition for a sequence of IS changes of measure as follows.
Jensen's inequality and the unbiasedness of $\hat{f}(1)$ implies
$$
\liminf_n \frac{1}{n} 
\log \hat{\mathbb E}[\hat{f}^2(1)] \ge
\liminf_n \frac{2}{n}  \log {\mathbb E}[\hat{f}(1)]\equiv 2\gamma.
$$
In other words, the exponential growth rate of the second moment of any
sequence of IS samples is at least
twice that of ${\mathbb E}[f_n].$
A sequence of IS estimators is said to be {\it asymptotically optimal} if
the lower bound is achieved, i.e., if
\begin{equation}\label{e:ao}
	\limsup_n \frac{1}{n}\log \hat{\mathbb E}[\hat f(1)^2] =
	\limsup_n \frac{1}{n}\log{\mathbb E}\left[f(1)^2\frac{dP}{d\hat{P}^*_n}
\right]
\le 2\gamma.
\end{equation}

\subsection{The IS problem for the Rao Bound}
In the context of estimating the expectation representation 
\eqref{e:expectation} of the Rao bound using IS, 
$f_n$ in \eqref{e:ldlimit} is
$$
f_n = 1_{\{S_n \le t/2\}}
\prod_{j=1}^{n} 2r(X_j,j) 
$$
where $S_j$ is the symmetric random walk with increments 
$X_j$ defined earlier.
The reason IS is necessary is because of the $1_{\{S_n \le t/2\}}$ term.
If we take $ t = \mu n$ with $\mu \le 1$,
as $n$ goes to $\infty$ the probability of $S_n$ being less than
$t/2$ goes to $0$ exponentially. 
In order to simulate  $X$ and $S$ using importance sampling one specifies
a sampling distribution $\bar{p}(v|i,s)$, $v \in \{0,1\} $ and
$s \in {\mathbb Z}_i $ and simulates $X$ from this distribution as
follows. One sets $S_0= 0$. At step $i$ of the simulation 
a random increment $X_i$ is sampled from the distribution 
$\bar{p}(\cdot|i, S_i)$ and sets $S_{i+1} = X_i + S_i$.
Note that the distribution of the increment $X_i$ is allowed to depend
on the current position of the random walk $S$.
Let $\bar{P}$ denote the probability measure on the sample paths of $S_n$
defined by the transition probability $\bar{p}(\cdot|i,s).$ Then
the Radon Nikodym derivative $\frac{dP}{d\bar{P}}$ equals
$\prod_{i=1}^n \frac{0.5}{\bar{p}(X^k_j|j,S^k_j)}$.
Then,
the IS estimator of ${\mathbb E}[f_n]$ using $K$ sample paths is
\begin{equation}\label{e:ISestimator}
\frac{1}{K} \sum_{k=1}^K \hat{f}_n(k),
       ~~~~~~ \hat{f}_n(k) \doteq  
1_{\{S^k_n \le t/2\}}
\prod_{j=1}^{n} \frac{r(X^k_j,j) }{\hat{p}(X^k_j|j,S^k_j)}
\end{equation}
where
$S^k$  denotes the $k^{th}$ independent sample path used in the simulation.
The increments $\{X^k\}$
are iid copies of the increment process $X$ sampled from $\bar{p}$. 
Then, by Theorem \ref{t:convergence}
the optimality condition \eqref{e:ao} for the IS estimator 
\eqref{e:ISestimator} is
\begin{equation}\label{e:ourao}
\limsup \frac{1}{n}\log 
{\mathbb E}\left[
1_{\{S_n \le t/2\}}
\prod_{j=1}^{n} 
\frac{2r(X_j,j)^2}
{\hat{p}(X_j|j,S_j)}\right] \le 2 V(0,0).
\end{equation}

\subsection{The limit optimization problem}
In the next subsection we will show that a sampling distribution
 $\bar{p}^*(\cdot|\cdot,\cdot)$
based on the large deviations analysis of the previous section satisfies
\eqref{e:ourao}, i.e., is asymptotically optimal.
It turns out that for the proof we won't need a complete asymptotic analysis
of the IS optimization problem \eqref{e:generalisinf}.
However, we include the following formal derivation of 
the limit optimization problem because it elucidates the direct connection
between IS and large deviations analysis.
As the reader will see, this connection is very general
and not limited to the current problem and has been known at least heuristically
for a long time, see for example \cite{ParWal} in the context of queuing networks.
A more rigorous and clear connection has been established recently in 
\cite{duphui-is1,duphui-is2, DSW, thesis, istrees}.

Now we proceed with our formal derivation.
For the present case, 
the IS optimization problem \eqref{e:generalisinf} becomes
\begin{align*}
&\inf_{\hat{p} } \log
{\mathbb E}\left[
1_{\{S_n \le t/2\}}
\prod_{j=1}^{n}
\frac{2 r(X_j,j)^2}
{\hat{p}(X_j|j,S_j)}\right].
\intertext{This equals}
&\inf_{\hat{p} } 
\sup_{\bar{p}: \bar{P}(S_n \le t_n / 2) } 
\bar{\mathbb E}\left[
\sum_{j=1}^{n} 2\log r(X_j,j) - \log \hat{p}(X_j|j,S_j) - \log\bar{p}(X_j| j, S_j) \right]
\notag
\intertext{by a direct generalization of Proposition \ref{p:relentrep}
to the present case.
It can be shown that this expression is convex in $\hat{p}$ and concave in $\bar{p}$ and therefore
the order of the $\inf$ and $\sup$ can be switched without effecting the result. Once this is done the optimization
in $\hat{p}$ gives the optimizer $\hat{p}^* = \bar{p}$ and the problem reduces to}
&
\sup_{\bar{p}: \bar{P}(S_n \le t_n / 2) } 
\bar{\mathbb E}\left[
\sum_{j=1}^{n} 2\log r(X_j,j) -  2\log\bar{p}(X_j| j, S_j) \right]
\notag
\end{align*}
and this is the same problem as in the representation \eqref{e:relentrep} except
for a factor of $2$. We know from the analysis of the previous section that when scaled
by $n$ this problem converges to
$$
\sup_{\theta(\cdot) } \int_{0}^1 [ 2\tilde{r}(t) \theta(s) + 2H(\theta(s)) ] ds,
$$
where the $\sup$ is over measurable $\theta > 0$ such that 
$\int_0^1 \theta(s)ds \le \mu/ 2.$
This is the same as \eqref{e:limit}, again except for a factor of $2$. Finally, and as before,
because $2\tilde{r}(t)\theta + 2H(\theta)$ is concave and independent of $t$ for $A_i \le t \le A_{i+1}$
the last problem reduces to
\begin{equation}\label{e:limitfinite2}
\sup 
\left\{ \sum_{i=1}^\sigma 
a_i\left( \theta_i 2\log (s_i -1)  + 2H(\theta_i) \right)\right\},
\end{equation}
where the $\sup$ subject to
\begin{equation*}
	\theta_i \in (0,1), ~~ \langle a , \theta \rangle = \mu/2.
\end{equation*}
Therefore the limit optimization problems for the large deviations analysis and
importance sampling are the same modulo a factor of $2$. In particular, the minimizers
$\theta^*$ of \eqref{e:limitfinite} are also the minimizers of \eqref{e:limitfinite2}.

\subsection{An asymptotically optimal IS sampling measure based on LD analysis}
There are many asymptotically optimal IS sampling measures to estimate
\eqref{e:expectation}. For example, one is $\bar{p}^*$ of \eqref{e:optimalcm}.
The problem with this  change of measure is that it requires the solution
of \eqref{e:const2} at every step of the random walk $S_n$. For large
$n$ this is inefficient. A much preferable situation is a fixed change
of measure, i.e., a change of measure $\bar{p}$ that doesn't depend
on $t$ and $x$. In the estimation of the Rao bound, we expect such a change of 
measure to exist for two reasons
1) the underlying process is iid and one dimensional 2) the probability of interest 
concerns exit from a region with only one boundary point. For more on these
points we refer the reader to \cite{duphui-is1,thesis} and \cite{Siegmund,ParWal,GlassKou}.
Let us now construct
a fixed change of measure for our problem.

Let 
$
\theta^*_i
$
to be the unique minimizers of \eqref{e:limitfinite} and define
\begin{equation}\label{e:isfixed}
\bar{p}^*(1|j,x) = \theta_i^*, \text{ if }  A_i \le j/n < A_{i+1}.
\end{equation}
$\bar{p}^*$ is almost fixed in the sense that it only
depends on the step number $j$ and not on the
position $x$ of the random walk $S_n$. The dependence on $j$ is very
intuitive and simple: each block of the orthogonal array has its own
fixed jump probability $\theta_i^*$, 
for the steps corresponding to the $i^{th}$ block 
one uses this fixed probability to sample the increments of $S_n$.

Before we state and prove our theorem
which asserts that an IS estimation based on \eqref{e:isfixed} 
is asymptotically optimal, we would like to make some comments and
setup several things that
we will need in the proof. Let us begin
with the computation of \eqref{e:isfixed}. One simply uses \eqref{e:explicit}
and \eqref{e:constraints} with $t=0$ and $x = 0$. Then
\begin{equation}\label{e:defthetai}
\theta_i^* = \frac{s_i - 1 } {e^{\lambda^*} + s_i -1 }
\end{equation}
where $\lambda^*$ is the unique solution of
\begin{equation}\label{e:const3}
\sum_{i=1}^\sigma a_j \frac{s_j - 1 } {e^\lambda + s_j - 1 } 
=\mu/2.
\end{equation}
Therefore, one can compute the IS change of measure 
$\bar{p}^*$ of \eqref{e:isfixed}
by simply solving the simple one dimensional problem \eqref{e:const3}
to identify $\lambda^*$ before the simulation begins. Throughout the
simulation no further computation will be necessary to calculate
$\bar{p}^*$. This is a great advantage over an IS simulation based 
on \eqref{e:optimalcm} which would require the solution of \eqref{e:const3}
at every step of the simulated random walk $S_j$.

\paragraph{Subsolutions}
A function $V$ that satisfies 
\begin{equation*}
 \sup_{\theta \in [0,1] } \{ \tilde{r}(t) \theta + H(\theta) + V_x(x,t) \theta
+ V_t(x,t) \} \ge 0
\end{equation*}
is called a {\em subsolution} to the PDE
\eqref{e:HJB}.
In the next paragraph we will construct a subsolution to 
\eqref{e:HJB} and the proof of asymptotic optimality will be a 
control theoretic verification argument based on this subsolution.
This technique is from the ``subsolution approach'' to IS
which was first developed
in the context of queuing networks in \cite{thesis, DSW}. For a more general
development see \cite{duphui-is3,duphui-is4}. The paper that precedes these
articles and which introduced many of the ideas that underlie the subsolution approach
is \cite{duphui-is1}. Other articles using the approach include 
\cite{istrees,MMpaperp,MR2358074}.

Usually,
the subsolution approach is very useful for constructing
good IS algorithms. 
This is the case in most of the aforementioned references.
In the present case, we already have a simple algorithm
and we will use the approach to prove that our algorithm is optimal.
For the subsolution, let us call it $W$, we set
$W_{x}(x,t) = -2\lambda^*$ for all $(x,t)$ and choose $W_t$
so that $W$ solves \eqref{e:HJB}:
\begin{equation}\label{e:Wt}
W_t(t,x) \doteq 
2\lambda^* \theta_i^*   -2\log(s_i - 1 ) \theta_i^* -2H(\theta_i^*),
~~~ A_i \le t < A_{i+1} .
\end{equation}
These define $W$ up to an additive constant.
This
is
sufficient for our needs since only the increments and partial derivatives 
of $W$ appear in a verification argument.
By its construction $W$ is 
piecewise affine,
continuous and
in fact a solution (and hence a subsolution)
to \eqref{e:HJB}.

\begin{remark}
{\em
$W$ is a solution to \eqref{e:HJB} and, as we have already noted in 
Theorem \ref{t:HJB}, so is
$V$ defined in \eqref{e:finitext}. Evidently $W\neq V$. This is a common
situation in optimal control, that is, an HJB equation may have many solutions.
What makes $V$ unique is that it is the maximal solution to \eqref{e:HJB}.
For more on these issues and a great deal of more information
on stochastic optimal control we refer the reader to \cite{Fleming}.
}
\end{remark}

Besides being a solution to \eqref{e:HJB} here are two properties of $W$ 
that play a key role in the optimality proof.
\begin{lemma}
\label{l:Wxlt0}
$W_x < 0$ and $W(\mu/2,1) - W(0,0) = 2V(0,0)$.
\end{lemma}
\begin{proof}
Let $g(\lambda)$ denote the left side of \eqref{e:const3}. 
$g$ is a decreasing function of $\lambda$, with $\lim_{\lambda\rightarrow
\infty} g(\lambda) = 0$ and $\lim_{\lambda\rightarrow -\infty} g(\lambda)= 1$.
$(s_i - 1 ) / s_i \ge 0.5$, because each $s_i$ is an integer greater
than $1$. Then $g(0) \ge 0.5 > \mu/2.$ It follows that $\lambda^* > 0$.
By definition $W_x = -2\lambda^* < 0$ and this is the first part of this lemma.

By their definition
$\theta_i^*$ satisfy
$\sum_{i=1}^\sigma \theta_i^*a_i = \mu/2$, see \eqref{e:defthetai}
and \eqref{e:const3}.
Let $x_j= \sum_{i=1}^j \theta_i^* a_i $.
One can write $W(\mu/2,1) - W(0,0)$ as the
following telescoping sum:
\begin{align*}
&W(\mu/2,1)- W(0,0)\\
 &~~= \sum_{i=1}^\sigma W(x_i,A_i) - W(x_{i-1},A_{i-1})
\intertext{By definition $W$ is affine for $t \in  (A_{i-1},A_i)$, with
partial derivatives $W_x= -\lambda^*$ and $W_t$ given in \eqref{e:Wt},
therefore this last sum equals}
                   &~~=\sum_{i=1} W_x(x_{i-1},A_{i-1}) (x_i - x_{i-1}) +
W_t ( x_{i-1},A_{i-1}) (A_i- A_{i-1} )\\
&~~= 
\sum_{i=1}^\sigma 2\lambda^* \theta^*_ia_i -2\lambda^* \theta_i^* a_i -2\log(s_i -1)\theta^*_i
-H(\theta_i^*)\\
&~~= 
- \sum_{i=1}^\sigma  2\log(s_i -1)\theta^*_i +2H(\theta_i^*)\\
\intertext{By definition $\theta_i^*$ are the unique 
optimizers of \eqref{e:limitfinite}, therefore}
&~~=
-2\sup\left \{ \sum_{i=1}^\sigma  \log(s_i -1)\theta_i +H(\theta_i) \right 
\}
\end{align*}
where the $\sup$ is subject to \eqref{e:orconst}. This last quantity
by definition equals $-V(0,0)$. This concludes the proof of the
second part of this lemma.

\end{proof}

It follows directly from the definitions of $W_t$ and $\theta_i^*$
that
\begin{equation}\label{e:multiplicative}
W_t+ \log \left( e^{W_x} (s_i-1)^2 \frac{1}{\theta^*_i} +\frac{1}{1-\theta^*_i}
\right) = 0.
\end{equation}
Let $X_i$ be a Bernoulli random variable with $P(X_i = 1) = 0.5$. 
For integers $x$ and $ A_{i-1} n \le j < A_i n$,
one can represent the previous display probabilistically as
\begin{equation}\label{e:onestep}
{\mathbb E}\left[ e^{W_x X_i +W_t(x/n, t/n) } (s_i-1)^{2X_i} \frac{2}{\bar{p}^*(X_i|x,j)}
\right] = 1
\end{equation}
\begin{remark}
{\em
The way it is presented above,
\eqref{e:multiplicative} seems unmotivated. One should
think of it as a multiplicative representation of \eqref{e:HJB}. 
One can derive \eqref{e:multiplicative} directly from \eqref{e:HJB}
first representing the optimization problem in that display as a 
trivial game and then using a representation result similar to 
\eqref{e:relentrep}. For a similar argument, see \cite[Lemma 2.5.2]{thesis}.
}
\end{remark}

\begin{theorem}\label{t:optimal}
The IS estimator based on $\bar{p}^*$ of \eqref{e:isfixed} 
is asymptotically optimal.
\end{theorem}
The following proof follows the same steps as the optimality proof
given in \cite{DSW}. It is simpler because there is a fixed time horizon
$n$ so no truncation of time is needed.
\begin{proof}
To ease notation let 
$\Delta W_i$ denote $W((S_{j+1}, (j+1))/n ) - W((S_j,j)/n ) )$.
Define
$$
M_k = 
\prod_{i=1}^{k-1} 
e^{n\Delta W_i} r(X_i,i)^2 \frac{2}{\bar{p}^*(X_i | i, S_i ) }.
$$
It follows from \eqref{e:onestep} that $M_k$ is a martingale and that
$$
{\mathbb E}\left[1_{\{S_n \le t_n /2 \} } 
\prod_{i=1}^{n-1} e^{n\Delta W_i} r(X_i,i)^2 \frac{2}{\bar{p}^*(X_i | i, S_i ) }\right] = 1.
$$
We saw in Lemma \ref{l:Wxlt0} that $W_x < 0$, therefore
{\small
$$\sum_{i=1}^{n-1} n\Delta W_i = 
n(W(S_n/n,1) - W(0,0)) > n(W(\mu/2,1) - W(0,0)) = -2nV(0,0)
$$
}
on $\{S_n < \mu n / 2\}$. The last two displays imply
$$
 e^{2nV(0,0)}
 \ge 
{\mathbb E}
\left[1_{\{S_n \le t_n /2 \} } 
\prod_{i=1}^{n-1} r(X_i,i)^2 \frac{2}{\bar{p}^*(X_i | i, S_i ) }\right].
$$
Taking the $\log$ of both sides, dividing by $n$ and letting
$n$ go to $\infty$
proves that \eqref{e:ourao} holds for the change of measure
$\hat{p}^*(\cdot|\cdot,\cdot)$. i.e., the IS change algorithm defined
by this change of measure is asymptotically optimal, which is what
we wanted to prove.
\end{proof}

\section{The Gilbert-Varshamov Bound}\label{s:GV}
The results derived for the Rao bound \eqref{e:rao1} in sections 
\ref{s:exp}, \ref{s:ld} and \ref{s:is} can be derived
for the Gilbert-Varshamov bound \eqref{e:gvbound}. The analysis and the results
are essentially the same, the main difference is that $\mu$ replaces $\mu/2$
in \eqref{e:constraintst0} and other similar places.

The key quantity in \eqref{e:gvbound} is
\begin{equation}\label{e:mainqgv}
\sum_{i=0}^{ t-1 } 
\sum_{
\mbox{
\begin{minipage}{2cm}
        \begin{center}\tiny
                        $u_1,u_2,\ldots,u_\sigma$\\
        $\sum u_m = i $
        \end{center}
\end{minipage}}} 
s_{\sigma}
\left( \begin{matrix} l_\sigma -1 \\ u_\sigma-1 \end{matrix}\right) 
(s_{\sigma} - 1)^{u_\sigma-1} 
\prod_{m=1}^{\sigma-1} 
\left(
\begin{matrix}
        l_m \\
        u_m
\end{matrix}
\right) (s_m - 1 )^{u_m}.
\end{equation}

Let $S_n$, $X_i$ and $r$ be defined as in Section \ref{s:exp}.
The expectation representation of \eqref{e:mainqgv} is
\begin{equation}\label{e:exp2}
s_\sigma {\mathbb E}
\left[ 
1_{\{ S_{n-1} \le t-1\} } \prod_{i=1}^{n-1} 2 r(X_i,i) \right].
\end{equation}
This is exactly the same as \eqref{e:expectation}, except for the
following differences. 
\begin{enumerate}
\item The expectation is over a random walk that takes
$n-1$ steps, rather than $n$,
\item There is a $s_\sigma$ factor in front,
\item The expectation is over those trajectories such that $S_{n-1} \le t-1$
rather than $S_n \le t/2$.
\end{enumerate}
As was the case in Section \ref{s:ld}
the asymptotic analysis of \eqref{e:exp2} will involve a $\frac{1}{n}\log$
scaling. Under this scaling 
the asymptotics of \eqref{e:exp2} is
the same as that of
\begin{equation}\label{e:exp2}
{\mathbb E}
\left[ 
1_{\{ S_{n} \le t\} } \prod_{i=1}^{n} 2 r(X_i,i) \right].
\end{equation}
Let $l_i$, $a_i$, $t_i$, 
$\mu$  be
as in \eqref{e:scalingassumptions}. Theorem \ref{t:convergence}
implies
	\begin{equation}\label{e:conv2}
\lim_n \frac{1}{n}\log
{\mathbb E}\left[ 1_{\{S_n \le t_n \}}
\prod_{j=1}^{n} 2r(X_j,j) \right] = 
\sup 
\left\{ \sum_{i=1}^\sigma 
a_i\left( \theta_i \log (s_i -1)  + H(\theta_i) \right)\right\},
\end{equation}
where the $\sup$ is over 
\begin{equation}\label{e:constraintslast}
	\theta_i \in (0,1), ~~ \langle a , \theta \rangle \le \mu.
\end{equation}
If $\mu \in (0.5,1)$ then $\{S_n \le \mu n \}$ is not a rare event
and there is no need for IS to simulate \eqref{e:exp2} effectively, one can use
straight forward Monte Carlo for this purpose.
Otherwise, Theorem \ref{t:optimal} implies that
the minimizers of \eqref{e:conv2} define an asymptotically optimal  IS change
of measure to estimate \eqref{e:exp2}.

\section{Numerical Results}\label{s:numerical}
We used the Octave numerical computation environment
\cite{octave} for the numerical computations in this
section.

\subsection{The Rao Bound}
\paragraph{Example 1.}
Consider the following parameter values for an orthogonal array:
$\sigma = 4$, 
alphabet sizes
$s$ = [ 13 10 7 5], the block lengths $l$ = [20 20 20 20] and $t = 4$. 
Then $n=80$ the scaled strength parameter $\mu = 0.05$, and length
parameters $a = [ 0.25 ~ 0.25 ~ 0.25 ~ 0.25]$.

For this example, the exact Rao bound can be computed in two ways: either using 
the original formula \eqref{e:rao1} or the recursive algorithm \eqref{e:recursive}. Both of these algorithms very quickly yield the value $190051$. 

We solve \eqref{e:limitfinite} with the above parameter values to get the large
deviation decay rate $V(0,0) = 	0.1681$. Then the large deviation estimate
of the Rao bound is $e^{V(0,0) n } = e^{13.44} = 689760$ which is about three times larger
than the actual bound found above. 
This type of inaccuracy is expected since an LD analysis only
identifies the exponential growth rate.

We know from Section \ref{s:is} that if the optimizers of \eqref{e:limitfinite} 
are used as an IS change of measure in \eqref{e:ISestimator} the resulting IS algorithm
is asymptotically optimal. The optimizers of \eqref{e:limitfinite} for the above
value of parameter values is $ \theta^* = (0.0383~    0.0290~    0.0195~    0.0131).$
Below are five estimation results using this algorithm with $K=2000$ sample paths. The
standard error column presents the estimated standard deviation $\hat{\sigma}(\hat{s}_K)$.
The informal $95\%$ confidence intervals are $[\hat{s}_K - 2\hat{\sigma}(\hat{s}_K) ~
\hat{s}_K + 2\hat{\sigma}(\hat{s}_K) ]$.
\begin{center}
\begin{tabular}{|l|c|c|c|c|}
\hline & Estimate $\hat{s}_K$ & Standard Error & 95 \% CI & Scaling\\
\hline
Est. 1   &1.94 &0.06 & [1.82 2.06] &  \\
Est. 2   &1.82 &0.06 & [1.70 1.94] &  \\
Est. 3   &1.83 &0.06 & [1.71 1.95] & $\times 10^5$ \\
Est. 4   &1.82 &0.06 & [1.70 1.94] & \\
Est. 5   &1.92 &0.06 & [1.80 2.04] & \\
\hline
\end{tabular}
\end{center}
The results in the table suggest that the asymptotically optimal IS scheme derived
in Section \ref{s:is} also perform well in practice. All of the estimates are close
to the actual value, the formal confidence intervals are tight and they all happen
to contain the exact Rao bound.

\paragraph{Example 2}
Now consider $\sigma = 40$, alphabet sizes $s_i = 20 + i$,
block lengths $l_i = 20$, $i = 1,2,...,40$ and strength parameter $t=20$. 
Then $n= 800$, $\mu=0.025$, and $a_i = 0.025$.

For this example, the complexity analysis \eqref{e:complex} 
in Section \ref{s:bounds} indicate that the direct computation of \eqref{e:rao1}
would require about $10^{41}$ operations, which of course is not possible
to perform in any reasonable amount of time. The recursive algorithm \eqref{e:recursive}
yields $2.57 \times 10^{38}$ in a second or less. Obviously this is an
impractically large value and it is clear that it is impossible to build an orthogonal
array for the parameters listed above. 

The large deviation decay rate $V(0,0)$ turns out to be $0.113$ for this problem
and the corresponding large deviation estimate of \eqref{e:rao1} is $ e^{V(0,0)n}
= e^{90.4} = 1.82 \times {10}^{38}$, which is, at the scale of $10^{38}$, close to
the exact value.

The optimizers of \eqref{e:limitfinite} is a forty dimensional vector and is inconvenient
to list explicitly. The IS estimate based on \eqref{e:rao1} using these optimizers
and $K=1000$ samples are as follows:
\begin{center}
\begin{tabular}{|l|c|c|c|c|}
\hline & Estimate $\hat{s}_K$ & Standard Error & 95 \% CI & Scaling\\
\hline
Est   &2.49 &0.14 &[2.21 2.77] & \\
Est   &2.58 &0.14 &[2.30 2.86] &\\
Est   &2.43 &0.14 &[2.15 2.71] & $\times 10^{38}$ \\
Est   &2.35 &0.14 &[2.07 2.63] &\\
Est   &2.55 &0.14 &[2.27 2.83] & \\
\hline
\end{tabular}
\end{center}
As in the first example, practical performance of the IS estimator is very good here.
All the estimates are close to the exact value, the confidence intervals are tight
and they all happen to contain the exact value.
The run time for each estimation is around a second.

\subsection{The Gilbert Varshamov Bound}
Let us continue with the
previous parameter values.
The computation for this bound is the same as Rao bound. 
In the example below, 
we calculate the expectation \eqref{e:exp2} rather than the
actual quantity \eqref{e:mainqgv}, which is a multiple of the expectation.
 We can use our recursive algorithm \eqref{e:recursive}
to compute the exact GV bound \eqref{e:gvbound} to be $3.13 \times 10^{71}$. The large deviation
growth rate $V(0,0) = 0.2088$ and the large deviation estimate of the GV bound is
$e^{V(0,0) 800} = 2.85 \times 10^{71}$.
The IS results are:
\begin{center}
\begin{tabular}{|l|c|c|c|c|}
\hline & Estimate $\hat{s}_K$ & Standard Error & 95 \% CI & Scaling\\
\hline
Est 1.   &3.5288 &0.23467 &[3.0595 3.9981] &  \\
Est 2.  &3.4698 &0.23083 & [3.0082 3.9315] &  \\
Est 3.  &3.4154 &0.2258 & [2.9638 3.867] & $\times 10^{71}$ \\
Est 4.  &3.2821 &0.21933 &[2.8434 3.7207] & \\
Est 5.  &2.8326 &0.19576 & [2.4411 3.2242] & \\
\hline
\end{tabular}
\end{center}
Once again they are accurate and reliable.
\subsection{A comparison}
A comparison of the asymptotic versions of the Rao and the GV bounds is given
Figure \ref{f:comparison}.
Take $q = 2$, $s_1 = 2$, $s_2 = 4$, $s_3 = 8$ and $s_4 = 16$,
$a_i = 0.25$.
The following graph depicts the Rao and the GV asymptotic bounds for
$\mu \in [0, 1].$
\ninseps{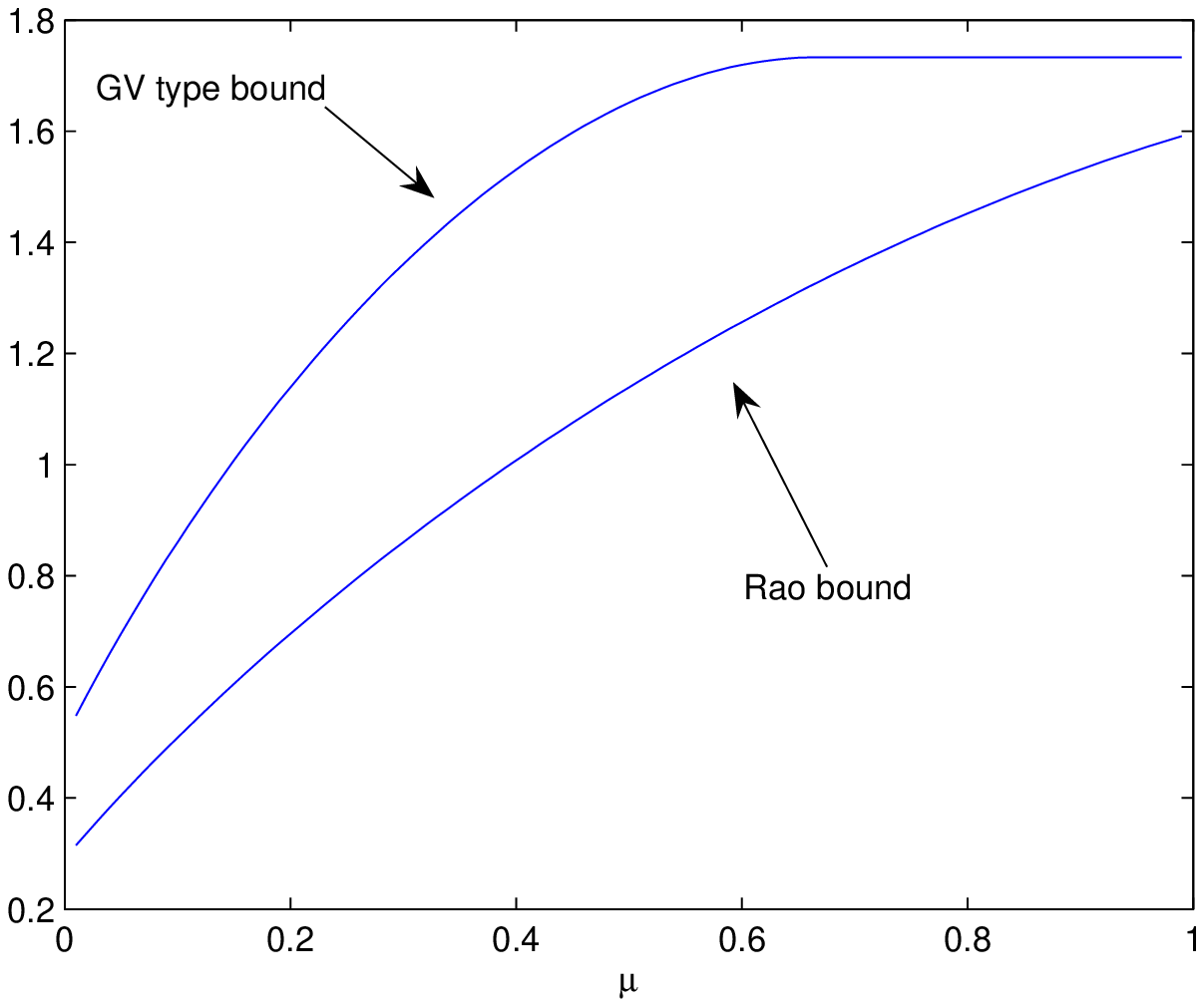}{A comparison of the Rao and GV bounds}{0.6}
The large
gap between them is due to the difference of a factor of $2$ between the
constraints of the Rao and the GV bounds.
The 
GV bound is
flat for larger values of $\mu$. This is because for these
values of $\mu$ the unique global maximizer
of \eqref{e:conv2} satisfies the constraint \eqref{e:constraintslast}.

\appendix

\end{document}